\documentclass[12pt]{amsart}
\usepackage[english]{babel}
\usepackage{amsmath}
\usepackage{amsthm}
\usepackage{amsfonts} 
\usepackage{hyperref}
\usepackage{epsfig}

\numberwithin{equation}{section}

\allowdisplaybreaks

\newtheorem{theorem}{Theorem}
\newtheorem{proposition}{Proposition}[section]

\newtheorem{remark}{Remark}

\newtheorem*{ack}{Acknowledgement}

\renewcommand{\epsilon}{\varepsilon}

\newcommand{\1}[1]{{\mathbf 1}{\{#1\}}}
\newcommand{\I}{\mathcal{I}}
\newcommand{\R}{\mathbb{R}}
\newcommand{\N}{\mathbb{N}}
\newcommand{\Z}{\mathbb{Z}}
\newcommand{\T}{\vec{T}} 
\newcommand{\PR}{\mathbb{P}}

\newcommand{\G}{\mathcal{G}}

\renewcommand{\phi}{\varphi}

\title[Local times of Markov Jump Processes]{Explicit formula for the density of local times of Markov Jump Processes}
\date{}
\author[R.~Huang]{Ruojun Huang}
\address{Ruojun Huang\\ Courant Institute of Mathematical Sciences, New York} \email{rh138@nyu.edu}
\author[D.~Kious]{Daniel Kious}
\address{Daniel Kious\\ NYU-ECNU Institute of Mathematical Sciences at NYU Shanghai} \email{daniel.kious@nyu.edu}
\author[V.~Sidoravicius]{Vladas Sidoravicius}
\address{Vladas Sidoravicius\\ NYU-ECNU Institute of Mathematical Sciences at NYU Shanghai, and Courant Institute of Mathematical Sciences, New York } \email{vs1138@nyu.edu}
\author[P.~Tarr\`es]{Pierre Tarr\`es}
\address{Pierre Tarr\`es\\ NYU-ECNU Institute of Mathematical Sciences at NYU Shanghai, Courant Institute of Mathematical Sciences, New York, and Centre National de la Recherche Scientifique, France} \email{tarres@nyu.edu}


\begin{document}

\begin{abstract}
In this note we show a simple formula for the joint density of local times, last exit tree and cycling numbers of continuous-time Markov Chains on finite graphs, which involves the modified Bessel function of the first type. 
  \end{abstract}

\maketitle
\section{Introduction}\label{intro}

Let $\G=(V,E,\sim)$ be a nonoriented connected finite graph with no multiple edges, with conductances $W_e\in\mathbb{R}_+\setminus\{0\}$, $e\in E$. For any $i\in V$, let $W_i=\sum_{j\sim i}W_{ij}$.

Consider the associated Markov Jump Process $(X_t)_{t\ge0}$ started at $i_0\in V$, that is the continuous-time discrete-space random walk which jumps from a vertex $i\in V$ to a neighbor $j$ at rate $W_{ij}=W_{ji}$, i.e. with  generator 
\[
\mathcal{L}f(i)=\sum_{j\in V: j\sim i} W_{ij}(f(j)-f(i))\text{, for any }i\in V.
\]

Let $\vec{E}=\{(i,j) : \{i,j\}\in E\}$ be the set of directed edges, where each undirected edge in E is replaced by two directed edges with opposite directions. For any oriented spanning tree $\vec{T}$, we call its root the unique site from which no edge goes out. We denote by $\delta_i(j)=\1{i=j}$ the Kronecker delta.

Let $\I$ be the set of currents on the graph, i.e.
$$\I=\{a\in\Z^{\vec{E}} \,:\,a_{ji}=-a_{ij},\,i,\,j\,\in V\}.$$
For any $a\in\Z^{\vec{E}}$ and $i\in V$, we let $a_i=\sum_{j\sim i}a_{ij}$. If $a\in\I$, then $a_i$ can be interpreted as the divergence of $a$ at site $i$. 

For any $k\in\N^{\vec E}$, let $a(k)\in\I$ be defined by $a(k)_{ij}=k_{ij}-k_{ji}$. For any $a\in\I$ and any oriented spanning tree $\vec{T}$ of $G$, let $\tilde{a}$ be defined by
\begin{align*}
\tilde{a}_{ij}&=a_{ij}-\1{ij\in \vec{T}}+\1{ji\in \vec{T}},\,\,ij\in\vec{E}.
\end{align*}

For any $\sigma>0$ and any right-continuous path $x=(x(t))_{t\ge0}$, let us define $\ell(x,\sigma)\in\R_+^V$ as the vector of local times at time $\sigma$, that is,
 $$\ell(x,\sigma)_i=\int_{0}^\sigma\1{x(s)=i}\,ds,\,\,i\in V.$$
Let us also define $k(x,\sigma)=(k_{ij}(x,\sigma))_{(i,j)\in\vec{E}}$ the vector of oriented crossings up to time $\sigma$, that is, 
\begin{align}
k_{ij}(x,\sigma)=|\{t\le\sigma: x_{t-}=i,x_t=j\}|. 
\end{align}

Let $\vec{T}(x,\sigma)$ be the last-exit tree of the path $x$ on the interval $[0,\sigma]$, that is, the collection of directed edges taken by path $x$ for the last departures from all vertices visited in that time interval  except the endpoint $x(\sigma)$. In other words, $(i,j)\in\vec T(x,\sigma)$ if there exists $t\in(0,\sigma]$ such that $(x_{t-},x_t)=(i,j)$ and $x(s)\neq i$ for every $s\ge t$.

For any $\nu\in\R$, the modified Bessel function of the first kind is defined by 
\begin{align}\label{bessel}
I_\nu(z)=\sum_{k=0}^\infty\frac{1}{k!\Gamma(k+\nu+1)}\left(\frac{z}{2}\right)^{2k+\nu},\,\,z\in\R.
\end{align}
Recall that $I_\nu(z)=I_{-\nu}(z)$. Therefore, for any $a\in\I$, any nonoriented edge $e=\{i,j\}\in E$ and $z\in\R$, we define $I_{a_{e}}(z)=I_{a_{ij}}(z)=I_{a_{ji}}(z)$. 

The main result of this note is the following.
\begin{theorem}\label{th1}
Let $i_0$, $i_1$ $\in V$, $\sigma>0$, $\ell\in(\R_+\setminus\{0\})^V$, let  $\vec{T}$ be an oriented spanning tree of the graph with root $i_1$, and let $a\in\I$ such that $a_i=\delta_{i_0}(i)-\delta_{i_1}(i)$ for all $i\in V$. Then 
\begin{align*}
\PR_{i_0}&\Big(a(k(X,\sigma))=a,\ell(X,\sigma)\in(\ell,\ell+d\ell), \vec{T}(X,\sigma)=\vec{T}\Big)\\
=&e^{- \sum_{i\in V}W_i\ell_i}
\prod_{e=\{i,j\}\in E}I_{\tilde{a}_{ij}}\left(2W_{ij}\sqrt{\ell_{i}\ell_{j}}\right)\prod_{(i,j)\in\vec{T}}W_{ij}\prod_{i\in V}\ell_i^{\tilde{a}_i/2}d\ell_i \, .
\end{align*}
\end{theorem}
\begin{remark}
It is easy to extend Theorem \ref{th1} to the case where $\sigma$ is a stopping time
$$\sigma=\sigma_u^{i}=\inf\{t\ge0:\,\,\ell_{i}(t)>u\},\,u>0,\,i\in V,$$
in which case we replace $\prod_{i\in V}d\ell_i$ by $\prod_{i\in V\setminus\{i_0\}}d\ell_i$, since we impose $\ell_{i_0}=u$.
In the case where $\G=\mathbb{T}$ is a (finite) tree and if $X_{0}=i_{0}$, at time $\sigma_u^{i_0}$ and on the event that $\ell_i(X,\sigma_u^{i})>0$, for all $i\in V$, there is only one choice of last exit tree $\vec{T}(X,\sigma)$, which is $\mathbb{T}$ itself oriented towards $i_0$, and on the other hand $a(k(X,\sigma))={\bf 0}$. Hence in that case $\PR_{i_0}(\ell(X,\sigma_u^{i_0})_{V\setminus\{i_0\}}\in(\ell,\ell+d\ell))$ equals
\begin{align*}
e^{- \sum_{i\in V}W_i\ell_i}
\prod_{e=\{i,j\}\in E}W_{ij}I_{1}\left(2W_{ij}\sqrt{\ell_{i}\ell_{j}}\right)\prod_{i\in V\setminus\{i_0\}}d\ell_i \, .
\end{align*}

This is consistent with the second Ray-Knight theorem that relates the local times of Brownian motion on $\R$ at time $\sigma_0^u$ to a $0$-dimensional squared Bessel process.
\end{remark}

\begin{remark}
One could obtain a formula for the density of the local times alone by summing the formula obtained in Theorem \ref{th1} over all possible spanning trees and cycling numbers. For this purpose, one should realize that, for any fixed oriented spanning tree $\vec{T}$, any arbitrary choice of cycling numbers on $\vec{E}\setminus\vec{T}$ can be uniquely extended to cycling numbers on the whole graph via a linear map.
\end{remark}

Explicit formulas for the joint density of local times of continous-time Markov Chains were already proposed, see for instance \cite{L,BHK}. Merkl, Rolles and Tarr\`es proposed in \cite{MRT} a formula for the joint density of the oriented edge crossings, local times and last-exit tree for the Vertex-Reinforced Jump Process on a general graph, whose counterpart in the context of continuous-time Markov Chains is stated in Proposition \ref{ppn:joint-density} below. 

Le Jan independently obtained in Theorem 4.1 \cite{LJ}, in the context of loop soups $\mathcal{L}_1$ with intensity $1$, an expression for the joint density of the cycling numbers and local time, which also involves the first modified Bessel function. We can deduce that result from the construction of those loop soups by Wilson's algorithm, in the following manner. 

Let us first quickly recall that algorithm: we order all the sites of our finite graph $V=\{i_0,\ldots,i_{|V|}\}$, and we assume that the walk is transient with cemetery $\Delta$. We start a loop-erased Markov chain $\{\eta_1\}$ starting from $i_0$ and ending at $\Delta$. Then, from the next vertex in $V\setminus\{\eta_1\}$ we start a loop-erased Markov chain $\{\eta_2\}$ ending in $\{\eta_1\}\cup\{\Delta\}$, and so on. The union of all $\eta_i$ is a spanning tree, whose leaves are the starting sites of the successive loop-erased chains. 

Given a fixed spanning tree $T$, we can easily obtain a formula similar to the one in Theorem \ref{th1} for the joint density of the succession of Markov chains starting successively at all leaves of $T$ with respect to the order on sites given above and killed at cemetery $\Delta$. Now the loop soup extracted from that succession of Markov chains by Wilson's algorithm has the same local time at all sites (see for instance Chapter 8 \cite{LJ0}), its cycling numbers are $k=\tilde{a}$ after extraction of the spanning tree; $\tilde{a}$ satisfies $\tilde{a}_i=0$ for all $i\in V$, so that the term $\prod_{i\in V}\ell_i^{\tilde{a}_i/2}$ in the density is $1$. Summing $\prod_{ij\in T}W_{ij}$ over all spanning trees of $G$ yields a determinant by matrix-tree theorem, which enables to deduce Theorem 4.1 \cite{LJ}. 

%%%%%%%%%%%%%%%%%%%%%%%%%%%%%%%%%%%%%%%%%%%%%%%%%%%%%%%%%%%%
%%%%%%%%%%%%%%%%%%%%%%%%%%%%%%%%%%%%%%%%%%%%%%%%%%%%%%%%%%%%
%%%%%%%%%%%%%%%%%%%%%%%%%%%%%%%%%%%%%%%%%%%%%%%%%%%%%%%%%%%%
%%%%%%%%%%%%%%%%%%%%%%%%%%%%%%%%%%%%%%%%%%%%%%%%%%%%%%%%%%%%
\section{Proof of Theorem \ref{th1}}
We first show the following Proposition \ref{ppn:joint-density}. Its proof relies on an argument similar to the proof of Theorem 1.6 in \cite{MRT}; the technique for determining the cardinality of the set of paths with given last exit tree are from Lemma 6 by Keane and Rolles in \cite{KR}.
\begin{proposition}\label{ppn:joint-density}
Let $i_0$, $i_1$ $\in V$, $\sigma>0$, $\ell\in(\R_+\setminus\{0\})^V$, let  $\vec{T}$ be an oriented spanning tree of the graph with root $i_1$, and let $k\in\N^{\vec{E}}$ be such that $a(k)_i=\delta_{i_0}(i)-\delta_{i_1}(i)$ for all $i\in V$. Then 
\begin{align}
&\PR_{i_0}\Big(
k(X,\sigma)=k,\ell(X,\sigma)\in(\ell,\ell+d\ell), \vec{T}(X,\sigma)=\vec{T}\Big)\nonumber\\
&=e^{- \sum_{i\in V}W_i\ell_i}\prod_{ij\in \vec{E}}\frac{(W_{ij}\ell_i)^{k_{ij}}}{k_{ij}!}\prod_{ij\in \vec{T}}\frac{k_{ij}}{\ell_i}\prod_{i\in V}d\ell_i \, .\label{eq:density}
\end{align}
\end{proposition}

\begin{proof}
It follows from a simple argument (similar but simpler than Lemma $1$ in \cite{KR}) that, for any $k\in\N^{\vec E}$ such that $a(k)=\delta_{i_0}-\delta_{i_1}$, there exists a path from $i_0$ to $i_1$ realizing the edge crossings prescribed by $k$. 
Consider adding to the event on the l.h.s. of \eqref{eq:density} the additional requirements that $(X_t)_{0\le t\le \sigma}$ takes a given path $\gamma=\{\gamma_0=i_0,\gamma_1,...,\gamma_{n-1},\gamma_n=i_1\}$ with given jump times $0<t_1<t_2<...<t_n<\sigma$, the probability turns out to be independent of such choices, and is equal to
\begin{align*}
\prod_{i\in V}\Big(W_i^{k_i}e^{- W_i\ell_i}\prod_{j\sim i}\Big(\frac{W_{ij}}{W_i}\Big)^{k_{ij}}\Big).
\end{align*}
The number of ways to distribute the jump times out of each vertex within its given local time contributes a multiplicative factor of
\begin{align*}
\Big(\prod_{i\neq i_1}\frac{\ell_i^{k_i-1}}{(k_i-1)!}\Big)\Big(\frac{\ell_{i_1}^{k_{i_1}}}{k_{i_1}!}\Big)\prod_{i\in V}d\ell_i.
\end{align*}
Further, with the last exit tree $\T$ fixed, the number of relative orders of exiting each vertex follows the multinomial distribution, thus contributing another factor of
\begin{align*}
\prod_{i\in V}\frac{(k_i-\1{i\neq i_1})!}{\prod_{j\sim i}(k_{ij}-\1{[i,j]\in \vec{T}})!}.
\end{align*}
Multiplying the preceding three displays proves the proposition.
\end{proof}

Let us now prove Theorem \ref{th1}. Let $a\in\I$ be such that $a_i=\delta_{i_0}(i)-\delta_{i_1}(i)$ for all $i\in V$.
For each nonoriented edge $e\in E$, let us choose a unique orientation $\vec{e}=(i,j)$ with $e=\{i,j\}$ such that $a_{ij}\ge0$, and let $E^+=\{\vec{e}:\,\,e\in E\}$. 

In order to compute the probability considered in the statement of the theorem, we need to sum all the contributions 
from Proposition \ref{ppn:joint-density} for all $k\in\N^{\vec E}$ such that $a(k)=a$. For each $ij\in E^+$, we sum over all $k_{ji}\ge0$ and $k_{ij}$ is determined by $k_{ij}=k_{ji}+a_{ij}\ge k_{ji}$. Therefore, using Proposition \ref{ppn:joint-density}, recalling that $\ell_i>0$ for every $i\in V$ and joining the contributions from $ij$ and $ji$ for each $ij\in E^+$, we have 
\begin{align*}
&\PR_{i_0}\Big(a(k(X,\sigma))=a,\ell(X,\sigma)\in(\ell,\ell+d\ell), \vec{T}(X,\sigma)=\vec{T}\Big)\\
&=e^{- \sum_{i\in V}W_i\ell_i}\sum_{(k_{ji})\in\N^{E^+}}\prod_{ij\in E^+}\frac{(W_{ij}\ell_j)^{k_{ji}}}{k_{ji}!}
\frac{(W_{ij}\ell_i)^{k_{ij}}}{k_{ij}!}
\prod_{ij\in \vec{T}}\frac{k_{ij}}{\ell_i}\prod_{i\in V}d\ell_i\\
&=e^{- \sum_{i\in V}W_i\ell_i}\prod_{ij\in E^+}\sum_{k_{ji}\ge\1{ji\in\T}}\frac{W_{ij}^{k_{ij}+k_{ji}}\ell_i^{k_{ij}-\1{ij\in\T}}\ell_j^{k_{ji}-\1{ji\in\T}}}{(k_{ij}-\1{ij\in\T})!(k_{ji}-\1{ji\in\T})!}\prod_{i\in V}d\ell_i
\end{align*}
In the second equality we use that, if $ji\in\vec{T}$, then the summand  is $0$ if $k_{ji}=0$. 

Let $k'_{ji}=k_{ji}-\1{ji\in\T}$. Then
\begin{align*}
k_{ij}-\1{ij\in\T}&=k'_{ji}+\tilde{a}_{ij}\\
k_{ij}+k_{ji}&=2k'_{ji}+\tilde{a}_{ij}+\1{ij\in T}
\end{align*}
where $T$ is the nonoriented tree associated to $\vec{T}$, so that 
\begin{align*}
&\sum_{k_{ji}\ge\1{ji\in\T}}\frac{W_{ij}^{k_{ij}+k_{ji}}\ell_i^{k_{ij}-\1{ij\in\T}}\ell_j^{k_{ji}-\1{ji\in\T}}}{(k_{ij}-\1{ij\in\T})!(k_{ji}-\1{ji\in\T})!}
=\sum_{k'_{ji}\ge0}\frac{W_{ij}^{2k'_{ji}+\tilde{a}_{ij}+\1{ij\in T}}\ell_i^{k'_{ji}+\tilde{a}_{ij}}\ell_{j}^{k'_{ji}}}{(k'_{ji}+\tilde{a}_{ij})!(k'_{ji})!}\\
&=W_{ij}^{\1{ij\in T}}\ell_i^{\tilde{a}_{ij}/2}\ell_j^{\tilde{a}_{ji}/2}	\sum_{k'_{ji}\ge0}\frac{(W_{ij}\sqrt{\ell_i\ell_j})^{2k'_{ji}+\tilde{a}_{ij}}}{(k'_{ji}+\tilde{a}_{ij})!(k'_{ji})!}=W_{ij}^{\1{ij\in T}}\ell_i^{\tilde{a}_{ij}/2}\ell_j^{\tilde{a}_{ji}/2}I_{\tilde{a}_{ij}}(2W_{ij}\sqrt{\ell_i\ell_j}).
\end{align*}
We conclude the proof by the observation that 
$$\prod_{ij\in E^+}\ell_i^{\tilde{a}_{ij}/2}\ell_j^{\tilde{a}_{ji}/2}=\prod_{i\in V}\ell_i^{\tilde{a}_i/2}.$$

%%%%%%%%%%%%%%%%%%%%%%%%%%%%%%%%%%%%%%%%%%%%%%%%%%%%%%%%%%%%
%%%%%%%%%%%%%%%%%%%%%%%%%%%%%%%%%%%%%%%%%%%%%%%%%%%%%%%%%%%%
%%%%%%%%%%%%%%%%%%%%%%%%%%%%%%%%%%%%%%%%%%%%%%%%%%%%%%%%%%%%

\begin{ack}
RH would like to thank the NYU-ECNU Institute of Mathematical Sciences at NYU Shanghai for its hospitality during his visit in Fall 2016, when most of  this work was done.
\end{ack}


\begin{thebibliography}{99}
%\bibitem{ACK} Angel, O., Crawford, N.~and Kozma, G.~(2014). Localization for linearly edge reinforced random walks. {\it Duke Math. J.} {\bf 163(5)}, 889--921.
%\bibitem{BP}Agresti, A.~(1975). On the Extinction Times of Varying and Random Environment Branching Processes. {\it J. of Appl. Probab.} {\bf 12(1)}, 39--46.
 \bibitem{BHK} Brydges, D., van der Hofstad, R.~ and K\"onig, W. (2007). Joint density of the local times of continuous-time Markov Chains. {\it Ann. Probab.} {\bf 35(4)}, 1307--1332.
% \bibitem{Coll} Collevecchio, A. (2006).  One the transience of processes defined on Galton-Watson trees. {\it Ann. Probab.} {\bf 34(3)}, 870--878.
%  \bibitem{CD} Coppersmith, D.~and Diaconis, P.~(1986). { Random walks with reinforcement}. {\it Unpublished manuscript.}
% \bibitem{Davis} Davis, B. (1990). { Reinforced random walk}. {\it Probab.~Theory and Related Fields}, {\bf 84(2)}, 203--229.
% \bibitem{DST} Disertori, M., Sabot, C.~and Tarr\`es, P.~(2015). Transience of edge-reinforced random walk. {\it Comm. Math. Phys.} {\bf 339(1)}, 121--148.
% \bibitem{DSZ} Disertori, M., Spencer, T.~and Zirnbauer, M.~R.~(2010). Quasi-diffusion in a 3D supersymmetric hyperbolic sigma model. {\it Comm. Math. Phys.} {\bf 200(2)}, 435--486.
% \bibitem{DoyleSnell} Doyle, P.G. and Snell E.J. (1984). { Random walks and electrical networks}. {\it Carus Math. Monographs}, {\bf 22}, Math. Assoc. Amer., Washington, D.C.
\bibitem{EKMRS} Eisenbaum, N., Kaspi, H., Marcus, M.B., Rosen, J. and Shi, Z. (2000).  A Ray-Knight theorem for symmetric Markov processes. {\it Ann. Probab.} {\bf 28(4)}, 1781-1796.
 \bibitem{KR} Keane, M.S.~and Rolles, S.W.W.~(2000). Edge-reinforced random walk on finite graphs. {\it Infinite dimensional stochastic analysis ({A}msterdam, 1999), Verh. Afd. Natuurkd. 1. Reeks. K. Ned. Akad. Wet.}. {\bf 52}, 217--234.
 \bibitem{LJ0} Le Jan, Y.~(2011). Markov paths, loops and fields. Lectures from the 38th Probability Summer School held in Saint-Flour, 2008. {\it Lecture Notes in Mathematics, 2026, Springern Heidelberg  }. 
 \bibitem{LJ} Le Jan, Y.~(2018). On Markovian random networks. {\it Preprint}.  arXiv:1802.01032.
% \bibitem{Kozsurvey} Kozma, G.~(2012). Reinforced random walks. {\it Preprint}. arXiv:1208.0364.
\bibitem{L} Luttinger, J.M. (1983). The asymptotic evaluation of a class of path integrals. II. {\it J. Math. Phys.} {\bf 24}, 2070--2073.
%  \bibitem{LyoPem} Lyons, R.~and Pemantle, R. (1992). Random Walk in a Random Environment and First-Passage Percolation on Trees. {\it Ann. Probab.} {\bf 20(1)}, 125--136.
% \bibitem{LP} Lyons, R. and Peres Y. (2016). Probability on trees and networks. Cambridge University Press. To appear,  \url{http://pages.iu.edu/\string~rdlyons/prbtree/prbtree.html}.
% \bibitem{MR} Merkl, F.~and Rolles, S.~W.~W.~(2009). Recurrence of the Edge-reinforced random walk on a two-dimensional graph. {\it Ann. Probab.} {\bf 37(5)}, 1679--1714.
\bibitem{MRT} Merkl, F.~and Rolles, S.~W.~W. and Tarr\`es~(2016). Convergence of vertex-reinforced jump process to an extension of the supersymmetric hyperbolic nonlinear sigma model {\it preprint}  arXiv:1612.05409.
%  \bibitem{Pemtree} Pemantle, R. (1988). Phase transition in reinforced random walk and RWRE on trees. {\it Ann. Probab.} {\bf 16}, 1229--1241.
%  \bibitem{Pemsurvey} Pemantle, R.~(2007). A survey of random processes with reinforcement. {\it Probab. Surv.} {\bf 4}, 1--79.
%  \bibitem{ST} Sabot, C.~and Tarr\`es, P.~(2015). Edge-reinforced random walk, vertex-reinforced jump process and the supersymmetric hyperbolic sigma model. {\it J. Eur. Math. Soc.} {\bf 17(9)}, 2353--2378.
%    \bibitem{ST2} Sabot, C.~and Tarr\`es, P.~(2015). Inverting Ray-Knight identity. {\it To appear in Probab.~Theory and Related Fields}.
%        \bibitem{STZ} Sabot, C., Tarr\`es, P.~and Zeng, X.~(2015). The Vertex Reinforced Jump Process and a Random Schr�dinger operator on finite graphs.{\it preprint}, 	arXiv:1507.04660.
%  \bibitem{SZ} Sabot, C.~and Zeng, X.~(2015). A random Schr�dinger operator associated with the Vertex Reinforced Jump Process on infinite graphs. {\it preprint}, 	arXiv:1507.07944.
%  \bibitem{Sellke} Sellke, T.~(2006). Recurrence of reinforced random walk on a ladder. {\it Electron. J. Probab.} {\bf 11}, 301--310.
%  \bibitem{PTsurvey} Tarr\`es, P.~(2011). Localization of reinforced random walks. {\it Preprint}. arXiv:1103.5536.
%    \bibitem{Vervoort} Vervoort, M. (2002). {\it Reinforced random walks}. 
\end{thebibliography}
\end{document}